\documentclass{article}
\usepackage{amsmath,amssymb,amsfonts,MnSymbol, mathrsfs,graphicx}
\usepackage{tikz,pgfplots,bm}
\usetikzlibrary{patterns}
\usetikzlibrary{shapes.geometric}

\def\e{{\rm e}}

\newtheorem{proposition}{Proposition}% 
\newtheorem{remark}{Remark}% 

\def\vI{{\mathbf{I}}}

\def\vF{{\mathbf{F}}}
\def\vP{{\mathbf{P}}}

\def\E{{\mathbb{E}}}

\def\Z{{\mathbb{Z}}}

\def\vk{{\mathbf{k}}}

\def\p{\partial}
\def\n{\nabla}

\def\d{\hbox{d}}
\def\One{{\bf 1}}

\def\vx{{\bf x}}

\def\vu{{\bf u}}

\newcommand{\ds}{\displaystyle}

\def\vom{{\bm\omega}}
\def\SS{{\mathbb{S}_2}}

\title{Collimated Sunlight and  Air Temperature}
\author{Olivier Pironneau\\ Sorbonne Université, Paris, France}

\begin{document}
\maketitle

\begin{abstract}
Atmospheric temperature on Earth results from complex phenomena that climate models must account for. One important module is radiation, which has three sources: infrared radiation emitted by the Earth and the air, and visible light from the Sun. The latter arrives from an almost point-like source in the sky, producing a Dirac singularity in the boundary conditions known as collimated light. Following Siewert and Maiorino (1980), we propose a very fast numerical implementation to handle the singularity.
\end{abstract}

\section*{Introduction}

In this article, we investigate the effect of collimated sunlight on a stratified atmosphere, focusing on how radiation influences air temperature.

Radiative transfer in the atmosphere involves two distinct sources: infrared emission from the Earth's surface and visible light from the sun. A key difficulty is that the sun subtends a nearly point-like angle in the sky, so its contribution enters as collimated light, proportional to 
$\delta(\vom-\vom_s)$ , the Dirac mass concentrating the radiation along the direction 
$\vom_s$  of the solar beam. To handle this singular source term, we follow the approach of Chandrasekhar \cite{CHA}, p240, and Siewert et al \cite{siewertM}.

A further complication concerns the vector nature of radiation. Polarization describes the orientation and evolution of the electric field associated with an electromagnetic wave, and although many radiative transfer models treat light as a scalar quantity, real atmospheric scattering processes modify not only the intensity of radiation but also its polarization state. This is particularly relevant to the atmosphere, since both Rayleigh scattering by air molecules and scattering by cloud droplets generate and transform polarized light. The resulting polarization state governs how radiation is redistributed in angle and intensity, and therefore directly affects the amount of energy absorbed within the atmosphere.

To account for these effects, we adopt the vector formulation of radiative transfer developed by Chandrasekhar \cite{CHA} and Pomraning \cite{POM2}, derived from the general equations for the Stokes vector \cite{CHA}, \cite{MISH} of quasi-plane EM waves.

For the numerical algorithm, we follow \cite{bookRTE} and reformulate it as a vector integral system, extended to handle the collimated solar source described above.

Numerical results indicate that GHG-augmented absorption leads to markedly different temperature responses depending on whether polarization is accounted for or not.

\section{Vector Radiative Transfer}

We first recall the governing equations coupling the Stokes vector of the radiation to the atmospheric temperature.

The electromagnetic radiation of interest for the atmosphere is a quasi-plane wave of very high frequency $\nu$, with a quasi-constant wave vector $\vk$. Equivalently, instead of $\vk$ and its intensity,  it can be characterized by its Stokes vector $\vI=[I,Q,U,V]^T$ and its direction $\vom$, i.e.\ the polar angle $\theta$ and azimuthal angle $\varphi$ on the unit sphere $\SS$.; $I$. is the radiative intensity and $Q,U,V$ is the polarization state.

When the refractive index is constant, the general Vector Radiative Transfer Equation (VRTE)  reads,
\begin{align}\label{VRTE2}&
\frac1c\p_t\vI+\vom\n_\vx \vI + \kappa\vI = \frac{\sigma_s}{4\pi}\int_\SS\Z(\vx,\vom',\vom)\vI(\vom')\d\omega' + \vF(T(\vx)),
\cr
&\rho c_V(\p_tT+\vu\cdot\nabla T)-\nabla\cdot({\mu_T}\nabla T)
=-{ T_0}{ B_0}\int_0^\infty\nabla\cdot\int_\SS I(\vx,\vom)\vom \d\vom \d\nu.
\end{align}
where $c$ is the speed of light, $\sigma_s$ the scattering coefficient, $\kappa$ the absorption coefficient, $\Z$ is the {phase matrix} for scattering from $\vom'$ to $\vom$, and $\vF$ is the source term due to atmospheric emission. 

The VRTE was formally derived from Maxwell's equations by Papanicolaou, Ryzhik and Keller \cite{papanico}. 

In the temperature equation for $T$ the right-hand side is the heat source due to the  radiation, the air density is $\rho$, $c_V$ is the specific heat capacity at constant volume, and $\mu_T$ is the thermal diffusivity; $T_0=4798$ and $B_0=1.744\times 10^6$ are scaling parameters. The wind velocity $\vu$ is either measured or obtained by solving the equations of meteorology.

Recall that a black body at temperature $T$ emits electromagnetic radiation at all frequencies $\nu$ proportionally to the Planck function.  If $T$ is scaled by $T_0$, $\nu$ by $\nu_0=10^{14}$ and Planck's function by $B_0$, then
 \[
 B_\nu(T)=\frac{\nu^3}{\e^{\frac{\nu}{T}}-1}, \qquad \hbox{ (Planck function, rescaled)}.
 \]
Thus $\vF(T(\vx))$ is proportional to $B_\nu(T(\vx))$, where $T(\vx)$ is  air temperature at $\vx$.

\subsection{The Stratified Case}\label{stratcase}
%%%%%%%%%%%%%%%%%%%%%%%%%%%%%%

If the physical domain $\Omega$ is the region between two infinite horizontal parallel planes, $z=0$ and $Z\approx 12$km, and if the ground and clouds are flat, then nothing depends on the horizontal coordinates $x,y$, and system \eqref{VRTE2} reduces to
\begin{align}\label{vrtegen}&
\mu\p_z \vI + \kappa\vI= \frac{\sigma_s}{4\pi}\int_0^{2\pi}\int_{-1}^1\Z(z,\mu',\mu,\varphi',\varphi)\vI(z,\mu',\varphi')d\mu' d\varphi' + \vF(T)
\cr&
\int_0^\infty \sigma_a B_\nu(T)  d\nu = \tfrac1{4\pi}\int_0^\infty\int_0^{2\pi} \int_{-1}^1 \sigma_a I d\mu d\nu 
\end{align}
where $\mu=\cos\theta$ and $\sigma_a=\kappa-\sigma_s$. The second equation is derived from the temperature equation when $\mu_T$ and $\vu$ are neglected; time $t$ has disappeared because $c\gg1$ and we neglect the temperature boundary layers.

In most cases, $\kappa=\rho\kappa_\nu$, where $\kappa_\nu$ depends only on $\nu$. A change of variables turns the altitude $z$ into the optical thickness, and $\rho$ disappears from the equations.
The atmosphere at temperature T generates radiation which according to Kirchhoff and Planck laws gives
\[
\vF(T)=[\sigma_a B_\nu (T),0,0,0]^T.
\]

\subsection{Boundary Conditions}\label{vectcollim}
%%%%%%%%%%%%%%%%%%%%%%%%%%%%%
Collimated light from the Sun, at angles defined by $\mu_s>0,\varphi_s$,  enters at the tropopause $Z$, while the infrared radiation at ground level is described by Lambert emission of intensity $c_e$ with Lambert albedo with coefficient $q_0$. This gives, for all  $\varphi\in(0,2\pi)$,
\begin{align}\label{bdc}&
I(Z,\mu,\varphi)= c_s B_\nu(T_s) \delta(\mu+\mu_s)\delta(\varphi-\varphi_s),  \quad \mu<0,
\cr&\ds
I(0,\mu,\varphi) = c_e\mu B_\nu(T_e)
+ \frac{q_0}{2\pi}\int_0^{2\pi}\int_{-1}^0 I(0,\mu',\varphi')\mu'\d\mu'\d\varphi', \quad  \mu>0.
\cr&
\end{align}  

\subsection{Azimuthal Independence}
%%%%%%%%%%%%%%%%%%%
As $\varphi$ appears only in $I(Z)$ we consider $\bar\vI$ defined by

\begin{align}\label{vrtegenb}&
\mu\p_z \bar\vI + \kappa\bar\vI = \frac{\sigma_s}2\int_{-1}^1\bar \Z(z,\mu',\mu)\bar\vI(z,\mu')d\mu'  + \vF(T).
\cr&
\int_0^\infty \sigma_a B_\nu(T)  d\nu = \int_0^\infty\tfrac12 \int_{-1}^1 \sigma_a \bar I d\mu d\nu 
\end{align}
where $\bar\Z$ is the $\varphi$-average of $\Z$ and with
\begin{align}\label{bdcb}&
\bar I(Z,\mu)= 0, \quad \mu<0,
\cr&\ds
\bar I(0,\mu) = c_e\mu B_\nu(T_e)
+ q_0\int_{-1}^0 \bar I(0,\mu')\mu'\d\mu', \quad  \mu>0,
\end{align}
and $\bar Q,\bar U,\bar V$  zero at $(z=0,\mu>0)\cup(z=Z,\mu<0)$.

If $\delta(\varphi-\varphi_s)$ was absent from \eqref{bdc} then a solution of  \eqref{vrtegenb}, \eqref{bdcb} is solution of \eqref{vrtegen} because the average over $\varphi'$ has the effect of replacing $\Z$ by $\bar\Z$.

\medskip

Let us isolate the azimuthal dependence of the solution of \eqref{vrtegen}.
\begin{proposition}
\[
\vI = \bar\vI + \vI'
\]
with $\vI'$ solution of
\begin{align}\label{vrtegenc}&
\mu\p_z \vI' + \kappa\vI'= \frac{\sigma_s}{4\pi}\int_0^{2\pi}\int_{-1}^1\Z(z,\mu',\mu,\varphi',\varphi)\vI'(z,\mu',\varphi')d\mu' d\varphi' 
\cr&
I'(Z,\mu,\varphi)= c_s B_\nu(T_s) \delta(\mu+\mu_s)\delta(\varphi-\varphi_s),  \quad \mu<0,
\cr&\ds
I'(0,\mu,\varphi) =  \frac{q_0}{2\pi}\int_0^{2\pi}\int_{-1}^0 I'(0,\mu',\varphi')\mu'\d\mu'\d\varphi', \quad  \mu>0.
\end{align}
\end{proposition}
\begin{proof}:
Note that $\bar\vI$  satisfies also \eqref{vrtegen} because 
\[
\frac12\int_{-1}^1\bar\Z\bar\vI\d\mu = \frac1{4\pi}\int_0^{2\pi}\int_{-1}^1\bar\Z\bar\vI\d\mu\d\varphi
=\int_\SS\Z(\vom,\vom')\bar\vI(\vom')\d\vom.
\]
All equations are linear; by adding those for $\bar\vI$ to those for $\vI'$, \eqref{vrtegen}, \eqref{bdc} are obtained.
\end{proof}
\begin{remark}
In the Stokes vector, $I$ is the radiative intensity, $Q$ the polarization intensity, and $U$, $V$ describe the polarization state. Alternatively, one may work with $\vI=[I_l,I_r,U,V]^T$, where $I_l=(I+Q)/2$ and $I_r=(I-Q)/2$. The equations are unchanged, but the phase matrix is now denoted $\Z'$, and the radiation volume source term becomes $\vF'(T) = \tfrac12\sigma_a B_\nu(T)[1,1,0,0]^T$.
\end{remark}
\subsection{Scaling}
All intensities  are divided by $B_0=1.744\times 10^6$. The temperature and the frequency are likewise scaled, respectively, by $T_0=4798$ and $\nu_0=10^{14}$ (see \cite{bookRTE}). Kirchhoff's law implies $\sigma_s=\kappa a_s$, $\sigma_a=\kappa(1-a_s)$, where $0\le a_s<1$ is the scattering parameter.
\subsection{The Rayleigh Phase Matrix}
%%%%%%%%%%%%%%%%%%%%
Following eq.(220) in \cite{CHA}, p42, the phase matrix for $[I_l,I_r,U,V]^T$ is
\begin{equation*}
\Z'\left(\mu, \varphi ; \mu^{\prime}, \varphi^{\prime}\right)
=\vP^{(0)}+(1-\mu^2)^\frac12(1-\mu')^\frac12 \boldsymbol{P}^{(1)}+\boldsymbol{P}^{(2)}\left(\mu, \varphi ; \mu^{\prime}, \varphi^{\prime}\right)
\end{equation*}
The top left $2\times 2$ blocks of these matrices are
\begin{equation}\label{P0}
\vP^0=\frac34\left(\begin{array}{cc}
2\left(1-\mu^2\right)\left(1-\mu^{\prime 2}\right)+\mu^2 \mu^{\prime 2} & \mu^2 \\
\mu^{\prime 2} & 1 \\
\end{array}
\right),
\end{equation}
\begin{equation}\label{P1}
\boldsymbol{P}^1=
\frac{3}{4} \cos \left(\varphi^{\prime}-\varphi\right) \left(\begin{array}{cc}
4 \mu \mu^{\prime}& 0 \\
0 & 0  \\
\end{array}\right),
\end{equation}
\begin{equation}\label{P2}
\boldsymbol{P}^2
=\frac{3}{4} \cos 2\left(\varphi^{\prime}-\varphi\right) \left(\begin{array}{cc}
\mu^2 \mu^{\prime 2}& -\mu^2 \\
-\mu^{\prime 2}  & 1
\end{array}\right).
\end{equation}
Hence, $\bar\Z'=P^0$. Pomraning \cite{POM} suggests using a linear combination of the Rayleigh matrix with the isotropic scattering matrix,
{%\scriptsize
\begin{equation}\label{Zp}
\bar\Z' =
\frac{3\beta}4\left[\begin{matrix}
	2(1-\mu^2)(1-\mu'^2) + \mu^2 \mu'^2 & \mu^2 \cr
	\mu'^2 & 1 
	\end{matrix}\right]
	+
\frac{1-\beta}2
\left[\begin{matrix}
1 & 1 \cr
1 & 1 
\end{matrix}\right].
\end{equation}

\section{Decoupling  the Stokes Vector Components}
%%%%%%%%%%%%%%%%%%%%%%%%%%%
The decomposition result of \cite{siewertM} applied to \eqref{vrtegenc}, implies the following.
\begin{proposition}
\begin{align}\label{siew}
\left[\begin{matrix}
I'_l \cr
I'_r
\end{matrix}\right](z,\mu,\varphi)
=
\left[\begin{matrix}
\tilde I'_l \cr
\tilde I'_r
\end{matrix}\right](z,\mu)
&
+ [(1-\mu^2)(1-\mu_s^2))]^{\frac12}\psi^1(z,\mu)\vP^1(\mu,\varphi,-\mu_s,\varphi_s)\vF''
\cr&
+  \psi^2(z,\mu)\vP^2(\mu,\varphi,-\mu_s,\varphi_s)\vF''
\cr&
+ \pi\delta(\mu+\mu_s)\e^{-\frac z\mu}\left[\delta(\varphi-\varphi_s)-\frac1{2\pi}
\right.\cr&\left.
-\frac2{3\pi(1+2\mu_s^2)}\vP^1(\mu,\varphi,-\mu_s,\varphi_s) 
\right.\cr&\left.
-\frac4{3\pi(1+\mu_s^2)^2}\vP^2(\mu,\varphi,-\mu_s,\varphi_s) 
\right]\vF''
\end{align}
where $\vF''(T) = \tfrac{c_s}2 B_\nu(T_s)[1,1]^T$, where $\psi^1,\psi^2$ are solutions of 
\begin{align*}&
\mu\p_z\psi^1+\kappa\psi^1=\frac{3\sigma_s \beta}8\int_{-1}^1(1-\mu'^2)(1+2\mu'^2)\psi^1(z,\mu')\d\mu',
\cr&
\psi^1(Z,\mu)|_{\mu<0} = \tfrac83[(1+2\mu_s^2)(1-\mu_s^2)]^{-1}\delta(\mu+\mu_s), \quad \psi^1(0,\mu)|_{\mu>0}=0,
\cr&
\mu\p_z\psi^2+\kappa\psi^2=\frac{3\sigma_s \beta}{16}\int_{-1}^1(1+\mu'^2)\psi^2(z,\mu')\d\mu',
\cr&
\psi^2(Z,\mu)|_{\mu<0} = \tfrac{16}3(1+\mu_s^2)^{-2}\delta(\mu+\mu_s), \quad \psi^2(0,\mu)|_{\mu>0}=0,
\end{align*}
and where $\bar\vI'=[\tilde I'_l,\tilde I'_r]^T$ is given by  
\begin{align}\label{vrtegend}&
\mu\p_z \bar\vI' + \kappa\bar\vI'= \frac{\sigma_s}{2}\int_{-1}^1\Z(z,\mu',\mu)\bar\vI'(z,\mu')d\mu' 
\cr&
\bar I'(Z,\mu)= c_s B_\nu(T_s) \delta(\mu+\mu_s)\delta(\varphi-\varphi_s),  \quad \mu<0,
\cr&\ds
\bar I'(0,\mu,\varphi) =  \frac{q_0}{2}\int_{-1}^0 I'(0,\mu')\mu'\d\mu', \quad  \mu>0.
\end{align}
\end{proposition}
The reader is sent to C. Siewert \& J. Maiorino \cite{siewertM} for the details.
\begin{remark}
The notation $\bar\vI'$ instead of $\tilde\vI'$ is justified because it is indeed the $\varphi$-mean of $\vI'$ since \eqref{siew} averaged contains only zeros except the first term on the right hand side. Indeed the averages of $\vP^1$ and $\vP^2$ are zero and 
\[
\frac1{2\pi}\int_0^{2\pi}[\delta(\varphi-\varphi_s)-\frac1{2\pi}]\d\varphi = 0.
\]
\end{remark}
\begin{remark}
We can aggregate $\bar\vI$ with $\bar\vI'$ and note that $\tilde\vI:=\bar\vI+\bar\vI'$ is given by 
\begin{align}\label{vrtegend}&
\mu\p_z \tilde\vI + \kappa\tilde\vI= \frac{\sigma_s}{2}\int_{-1}^1\Z'(z,\mu',\mu)\tilde\vI(z,\mu')d\mu' +\vF'(T)
\cr&
\tilde\vI(Z,\mu)=  \delta(\mu+\mu_s)\vF'' , \quad \mu<0,
\cr&\ds
\tilde \vI(0,\mu,\varphi) =  \frac{c_e}2 \mu B_\nu(T_e)[1,1]^T + \frac{q_0}{2}\int_{-1}^0 \tilde \vI(0,\mu')\mu'\d\mu', \quad  \mu>0.
\end{align}
\end{remark}
\begin{proposition}
The temperature in the atmosphere is given by  \eqref{vrtegend} which is the VRTE for $\tilde I$  with semi-collimated light   input at the tropopause.
\end{proposition}
\subsection{Integral formulation}
%%%%%%%%%%%%%%%%%%%%%%%%%%%%%
Recall that if $I_\nu$ satisfies  $\mu\partial_z I_\nu+\kappa_\nu I_\nu = S$, the method of characteristics  tells us that
\begin{align}\label{I9b}
	I_\nu(z,\mu) &= \One_{\mu>0}\left[ \e^{-\frac1\mu \kappa_\nu z }I_\nu(0,\mu)
	+ \int_{0}^z \frac1\mu\e^{-\frac1\mu \kappa_\nu(z-z')}S(z')\d z'\right]
\cr&
	+ \One_{\mu<0}\left[ \e^{\frac1\mu \kappa_\nu(Z-z) }I_Z(\mu)
	- \int_{z}^Z \frac1\mu\e^{\frac1\mu \kappa_\nu(z'-z)}S(z')\d z'\right].
\end{align}
For $\psi^1$
\begin{align}&
S^1= \frac34\beta\sigma_s[\frac12\int_{-1}^1(1-\mu^2)(1+2\mu^2)\psi^1(\mu)\d\mu]
=  \frac34\beta\sigma_s[J^1_0+J^1_2-2 J^1_4]
\cr&
\hbox{where }~J^1_q = \frac12\int_{-1}^1\mu^q\psi^1\d\mu.
\end{align}
Using the boundary conditions and   \eqref{I9b} multiplied by $\mu^q$ and integrating in $\mu$ leads to
\begin{equation}\label{jq}
J^1_q(z) = \frac{8\e^{-\kappa_\nu\frac{Z-z}{\mu_s}}}{3(1+2\mu_s^2)(1-\mu_s^2)}
+ \frac12 \int_0^Z  E_{q+1}(\kappa_\nu|z-z'|)S^1\d z'
\end{equation}
where $E_q$ is the $q^{th}$ exponential integral.
For  $\psi^2 $
\begin{align}&
S^2= \frac{16}3\beta\sigma[\frac12\int_{-1}^1(1+\mu^2)(\psi^2(\mu)\d\mu]
=  \frac{16}3\beta\sigma[J^2_0+J^2_2]
\cr&
J^2_q = \frac12\int_{-1}^1\mu^q\psi^2\d\mu =  \frac{16\mu_s^q\e^{-\kappa_\nu\frac{Z-z}{\mu_s}}}{3(1+\mu_s^2)^2}
+ \frac12\int_0^Z [E_{q+1}(\kappa_\nu|z-z'|)S^2\d z'
\end{align}
To compute pointwise values of $\psi^1$ or $\psi^2$ use \eqref{I9b}. 

 Some integrals are singular. Claude.ai from Anthropic proposed the following formula which it claims to be almost always second order accurate. It also provides the C++ implementation.
 
 Let $0=x_0<x_1<\dots<x_N=X$, $f_i=f(x_i)$, and
\[
I(y)=\int_0^X f(x)\,\frac{e^{-|x-y|/m}}{m}\,dx
\;\approx\;
\sum_{i=0}^{N-1}\int_{x_i}^{x_{i+1}} \Pi_h f(x)\,\frac{e^{-|x-y|/m}}{m}\,dx ,
\]
where $\Pi_h f$ is the piecewise linear interpolant of $f$.
Each cell integral is computed exactly. For a cell $[a,b]$, $h=b-a$:

\medskip
\emph{Cell left of $y$} ($b\le y$), with $\alpha=e^{-(y-a)/m}$, $\beta=e^{-(y-b)/m}$:
\[
\int_a^b \Pi_h f\,\frac{e^{-(y-x)/m}}{m}\,dx
= f_a\!\left[\frac{m}{h}(\beta-\alpha)-\alpha\right]
+ f_b\!\left[\beta-\frac{m}{h}(\beta-\alpha)\right].
\]

\emph{Cell right of $y$} ($a\ge y$), with $\alpha=e^{-(a-y)/m}$, $\beta=e^{-(b-y)/m}$:
\[
\int_a^b \Pi_h f\,\frac{e^{-(x-y)/m}}{m}\,dx
= f_a\!\left[\alpha-\frac{m}{h}(\alpha-\beta)\right]
+ f_b\!\left[\frac{m}{h}(\alpha-\beta)-\beta\right].
\]

\emph{Cell containing $y$}: split at $y$, set $f_y=f_a+(f_b-f_a)\frac{y-a}{h}$,
and apply the two formulas above on $[a,y]$ and $[y,b]$.

\medskip
The error satisfies
\[
|I(y)-I_h(y)|\le \frac{h^2}{8}\,\|f''\|_\infty \int_0^X \frac{e^{-|x-y|/m}}{m}\,dx
\le \frac{h^2}{4}\,\|f''\|_\infty ,
\]
uniformly in $m$; the rule is exact for $f\equiv 1$, giving
$2-e^{-y/m}-e^{-(X-y)/m}$.
 
\subsection{Iterations on the Source}
%%%%%%%%%%%%%%%%%%

It is shown in \cite{bookRTE} that the following iterations to compute the moments of $\psi^1$ are monotone and convergent,
\begin{enumerate}
\item Initialize $S$.
\item Compute $J_q$ by \eqref{jq}
\item Update $S$ by $S=  \frac34\beta\sigma[J_0+J_2-2 J_4]$.
\end{enumerate}
After a few iterations $S(z)$ will be computed and if $q_0=0$, $\psi^1$ can be displayed by
\[
\psi^1=  \One_{\mu<0}\e^{\frac1\mu \kappa_\nu(Z-z) }c_s B_\nu(T_s)
	+ \int_{0}^Z \frac1{|\mu|}\e^{-\frac1{|\mu]} \kappa_\nu|z'-z|}S(z')\d z'.
\]
If $q_0<0$ a slightly more complex formula can be found in \cite{bookRTE}.

Naturally the same procedure can be applied to compute $\psi^2$

\subsection{Computation of $ \tilde I$ and $\tilde Q$}
%%%%%%%%%%%%%%%%%%

A system of equations is derived easily for $(\tilde I,\tilde Q)$ by linear combinations  in \eqref{vrtegend}
\begin{align}\label{lq0}\ds &
\mu \p_z \tilde I + \kappa \tilde I =\sigma_a B_\nu(T) + \frac{\sigma_s}2\int_{-1}^1 \tilde I\d\mu'
%\cr&&
+ \frac{\beta\sigma_s}4 P_2(\mu)\int_{-1}^1 [P_2 \tilde I-(1-P_2 )\tilde Q]\d\mu',
\cr&
\mu \p_z \tilde Q + \kappa \tilde Q = -\frac{\beta\sigma_s}4 (1-P_2(\mu))\int_{-1}^1 [P_2 \tilde I-(1-P_2 )\tilde Q]\d\mu',
\cr&
\tilde I(Z,\mu)=  \delta(\mu+\mu_s) c_s B_\nu(T_s), \quad \tilde Q(Z,\mu)=0,~~ \mu<0, 
\cr&\ds
\tilde I(0,\mu,\varphi) =  c_e \mu B_\nu(T_e) + q_0\int_{-1}^0 \tilde I(0,\mu')\mu'\d\mu', \quad \tilde Q(0,\mu)=0,~ \mu>0,
\end{align}
where  $P_2(\mu)=\tfrac12(3\mu^2-1)$. 

The integral formulation follows.
Denote
\begin{align}\label{JKK}
	J_q(z) = \tfrac12\int_{-1}^1 \mu^q\tilde I\d\mu
\quad
	K_q(z) = \tfrac12\int_{-1}^1 \mu^q\tilde Q\d\mu, 
\quad
	K_2(z) = \tfrac12\int_{-1}^1 \mu^2\tilde Q\d\mu .
\end{align}
Then, 
\begin{align}\label{iqpde}\ds 
	\mu \p_z \tilde I + \kappa \tilde I =&\sigma_a B_\nu + \sigma_s J_0
	+ \frac{\beta\sigma_s}4 P_2(\mu)(3 J_2 - J_0 -3 K_0 + 3 K_2),
\cr
	\mu \p_z \tilde Q + \kappa \tilde Q 
	= &-\frac{\beta\sigma_s}4 (1-P_2(\mu))(3J_2-J_0 - 3 K_0 +3 K_2).
\end{align}
By the method of characteristics and Proposition 1.10 in \cite{bookRTE} ,
\begin{align}\label{JK}\ds
J_q(z) &= R_q +\tfrac{1}2\int_0^Z \left(\E_{q+1}(\kappa_\nu|z-z'|)S_0(z')+ \E_{q+3}(\kappa_\nu|z-z'|)S_2(z')\right)\d z',
\cr
K_q(z) &=\tfrac12\int_0^Z \left(\E_{q+1}(\kappa_\nu|z-z'|)S'_0(z') + \E_{q+3}(\kappa_\nu|z-z'|)S'_2(z')\right)\d z',
\end{align}
with $c_0$ and $R_q$ given by
\begin{align}\label{c0}
&c_0 = - q_0 \mu_s c_s B_\nu(T_s)\e^{-\frac{\kappa_\nu Z}{\mu_s}}
\cr&
~~~~~ -q_0 \int_0^Z \left\{\E_2(\kappa_\nu z)\left[ \sigma_a B_\nu(T) +  \sigma_s J_0-\tfrac13 H
\right]+\E_4(\kappa_\nu z)H\right\}\d z,
\end{align}
\begin{align}\label{SS0}&
R_q(z,\nu) = \tfrac{c_e}2  B_\nu(T_e) \E_{q+3}(\kappa_\nu z)  + \tfrac{c_s}2  B_\nu(T_s)\mu_s^q\sigma_s\e^{-\kappa_\nu\frac{Z-z}{\mu_s}}+
\tfrac{c_0}2 \E_{q+2}(\kappa_\nu z),
\end{align}
Denote $H(z,\nu) = \frac{9\beta\sigma_s}8(J_2-\tfrac13 J_0 - K_0 + K_2)$. Then
\begin{align*}&
S=S_0+\mu^2 S_2, \quad S'=S'_0+\mu^2 S'_2,
\cr& 
S_0=\sigma_a B_\nu + \sigma_s J_0 - \frac13 H,
\qquad
S_2 = H,
\qquad%\cr&
S'_0=-H,
\qquad
S'_2 = H.
\end{align*}
So at each iteration we only need to compute, for $q=0,2$,
\begin{align}\label{JKq}
J_q(z) &= R_q +\tfrac12 \int_0^Z \left(\E_{q+1}(\kappa_\nu|z-z'|)(\sigma_a B_\nu + \sigma_s J_0 - \frac13 H) +  \E_{q+3}(\kappa_\nu|z-z'|)H \right)\d z',
 \cr
 K_q(z) &=  -\tfrac12 \int_0^Z \left(\E_{q+1}(\kappa_\nu|z-z'|) - \E_{q+3}(\kappa_\nu|z-z'|)\right)H\d z',
 \end{align}
 The Iterations on the source using the integral formulation (ISIF) applies to the vector $[\tilde I,\tilde Q]^T$.
 \begin{enumerate}
 \item Initialize $J_q,K_q,q=0,2$.
 \item Compute $c_0$ with \eqref{c0} and $R_q,~ q=0,2$ with \eqref{SS0}.
 \item Compute $T$  by solving the temperature equation in \eqref{vrtegenb} with dichotomy and Newton iterations.
 \item Update $J_q,K_q,~q=0,2$ with \eqref{JKq}.
 \end{enumerate}

Then $\tilde I$ and $\tilde Q$ are found by
\begin{align}\label{charac}&
\tilde I(0,\mu) |_{\mu>0}= c_e B_\nu(T_e)\mu - q_0\mu_s\e^{-\frac{\kappa_\nu Z}{\mu_s}}
- q_0\int_0^Z [E_2(\kappa_\nu z) S_0(z)+E_4(\kappa_\nu z) S_2(z)]\d z,
\cr&
\tilde I(z,\mu) = I(0,\mu)\e^{-\frac{\kappa_\nu z}{\mu} }+ \int_0^z \e^{-\frac{\kappa_\nu (z-z')}{\mu} }\frac{S(z',\mu)}\mu\d z', ~~\mu>0,
\cr&
\tilde I(z,\mu) = \delta(\mu+\mu_s)c_s B_\nu(T_s)e^{-\frac{\kappa_\nu(Z-z)}{\mu_s}} + \int_z^Z \e^{\frac{\kappa_\nu (z'-z)}{\mu} }\frac{S(z',\mu)}\mu\d z', ~~\mu<0,
\cr&
\tilde Q(z,\mu) =\int_0^z \e^{-\frac{\kappa_\nu (z-z')}{\mu} }\frac{S'(z',\mu)}\mu\d z', ~~\mu>0,
\cr&
\tilde Q(z,\mu) =  \int_z^Z \e^{\frac{\kappa_\nu (z'-z)}{\mu} }\frac{S'(z',\mu)}\mu\d z', ~~\mu<0,
\end{align}

\section{Numerical results}
%%%%%%%%%%%%%%%

The radiation intensities and albedo are computed with 
\[
c_e=2,\quad c_s=2\times 10^{-6}, \quad q_0=-0.3,\quad T_e=18^oC, \quad T_s =5800K.
\]
The collimated light from the sun comes with $\mu_s=0.5$ and $\varphi_s=0$.
The phase matrix is built with $\beta=0.5$.
The density is $\rho=1-0.7z$.
The scattering coefficient is
\[
a_s = 0.3   + 0.3*16(\nu<\nu_2)(\nu>\nu_1)((\nu-\nu_1)(\nu-\nu_2)/(\nu_1-\nu_2)^2)^2 (z>z_3)
    \]
    with $z_3=0.8$, $\nu_1=0.5,~\nu_2=1$.
  $\kappa_\nu$ is read from the Gemini telescope website \cite{gemini} compacted into 554 frequency readings. A portion of the spectrum is displayed in figure \ref{gemini}
  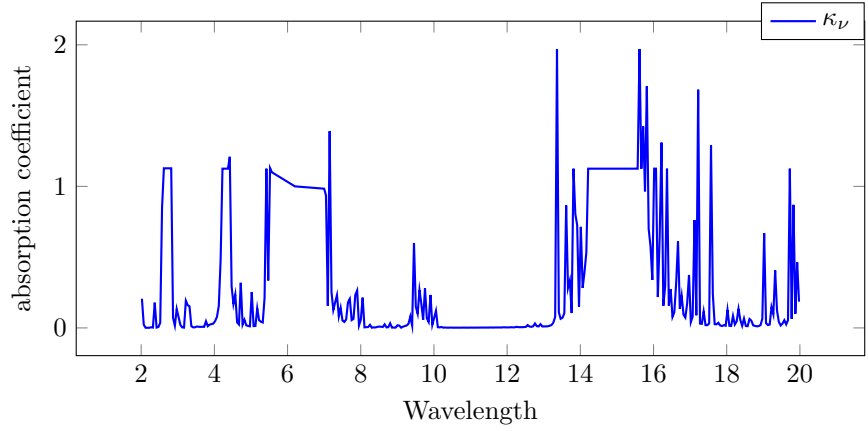
\begin{figure}[htbp]%gemini 
\begin{center}
\begin{tikzpicture}%[scale=0.55]
\begin{axis}[width=12cm, height=6cm, legend style={at={(1,1)},anchor=east}, compat=1.3,
   ylabel= {absorption coefficient},
  xlabel= {Wavelength}
  ]
\addplot[thick,solid,color=blue,mark=none, mark size=1pt] table [x index=0, y index=1]{kappafig.txt};
\addlegendentry{$\kappa_\nu$}
\end{axis}
\end{tikzpicture}
\caption{\footnotesize  \label{gemini} Absorption $\kappa_\nu$ versus wavelength ($c/\nu$) in the range $(2,20)$  (the computational rangl is $(0.0025,250)$ )read from the Gemini data site \cite{gemini}. }
\end{center}
\end{figure}

    There are 50 discretization points on $[0,Z]$.  The $\mu$-integrals are computed with 40 quadrature points.
    The computing time on a MacBook Pro arm64 is 0.25 second. It is remarkably fast in view of the fact that at each of thee 10 iterationss all 4 PDEs are solved 554 times and the temperature equation is solved 50 times.
    
    \medskip
    
    Figure \ref{tempe} shows the temperature in the atmosphere after 9 iterations.  The cusp  near $z=0$ is due to the Lambert albedo.  Figure \ref{psi12} shows $\psi^1$ and $\psi^2$ at a given frequency versus altitude and cosine of the polar angle.
 \begin{figure}[htbp]
%%%%%%%%%%%%%%%%%%%%
\begin{minipage} [b]{0.45\textwidth}. 
\begin{center}
\begin{tikzpicture}[scale=0.7]
\begin{axis}[legend style={at={(0.97,0.9)},anchor= east}, compat=1.3,
   xlabel= {Altitude / 10km},
  ylabel= {Temperature $^o$C}
  ]
\addplot[thick,solid,color=blue,mark=+, mark size=1pt] table [x index=0, y index=1]{tempe.txt};
\end{axis}
\end{tikzpicture}
\caption{\footnotesize \label{tempe} Temperature versus altitude.}
\end{center}
%%%%%%%%%%%
\end{minipage}
\hskip0.5cm
\begin{minipage}[b]{0.45\textwidth}
\begin{center}
\includegraphics[width=7cm]{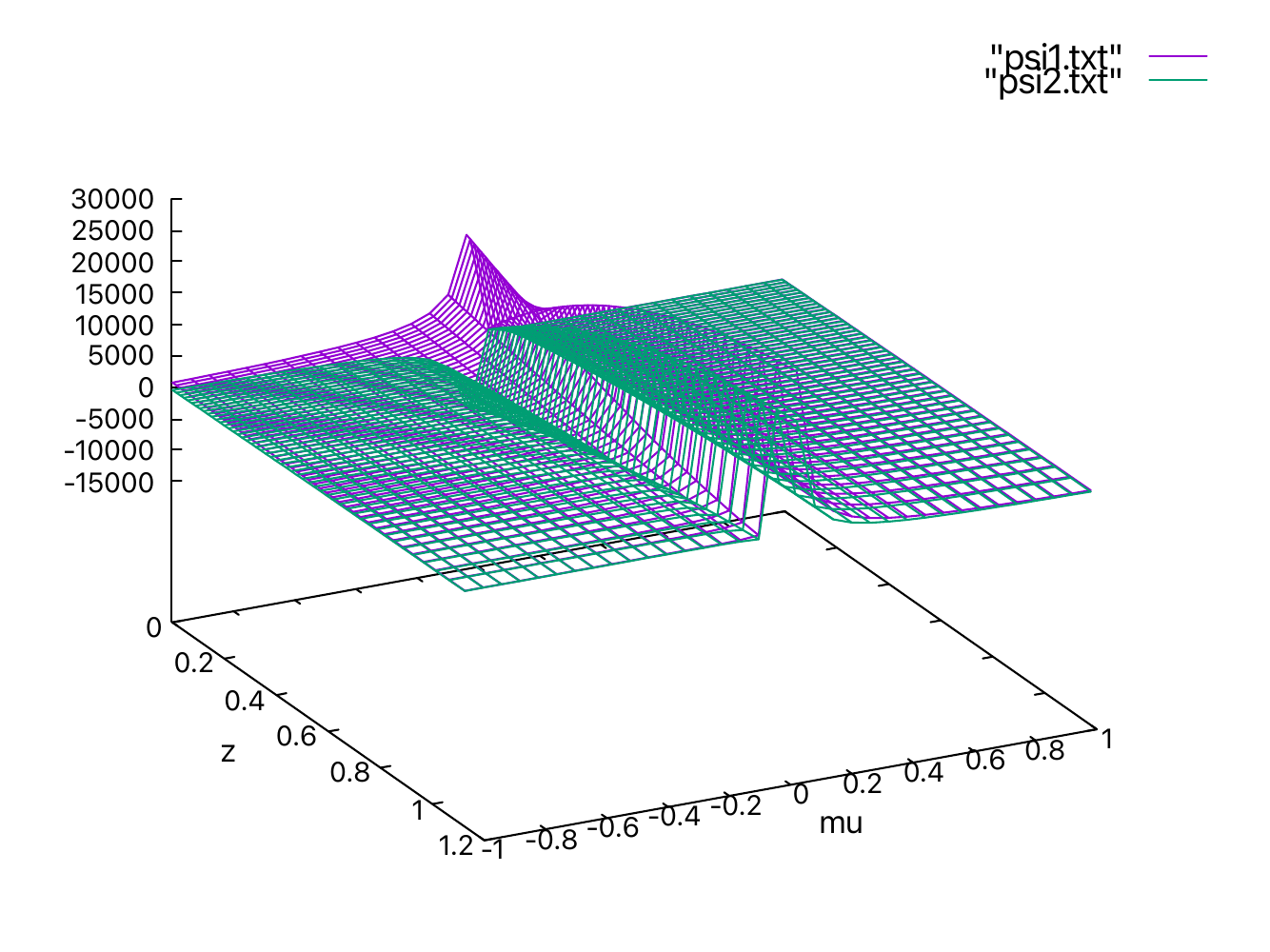}
\caption{\label{psi12}\footnotesize  $\psi^1$ and $\psi^2$  at  $\nu=4.2$ which is just outside the visible light range towards the infrared change, thence giving a chance to see both radiation input, the collimated light at $Z$ and the Lambert radiation at $z=0$.}
\end{center}
\end{minipage}
\end{figure}

Finally, figure \ref{Isurf} and figure \ref{Qsurf} show the radiative intensity $\tilde I$ and the polarization $\tilde Q$  versus altitude and cosine of the polar angle, $\mu$.
\begin{figure}[htbp]
%%%%%%%%%%%%%%%%%%%%
\begin{minipage} [b]{0.45\textwidth}. 
\begin{center}
\includegraphics[width=6cm]{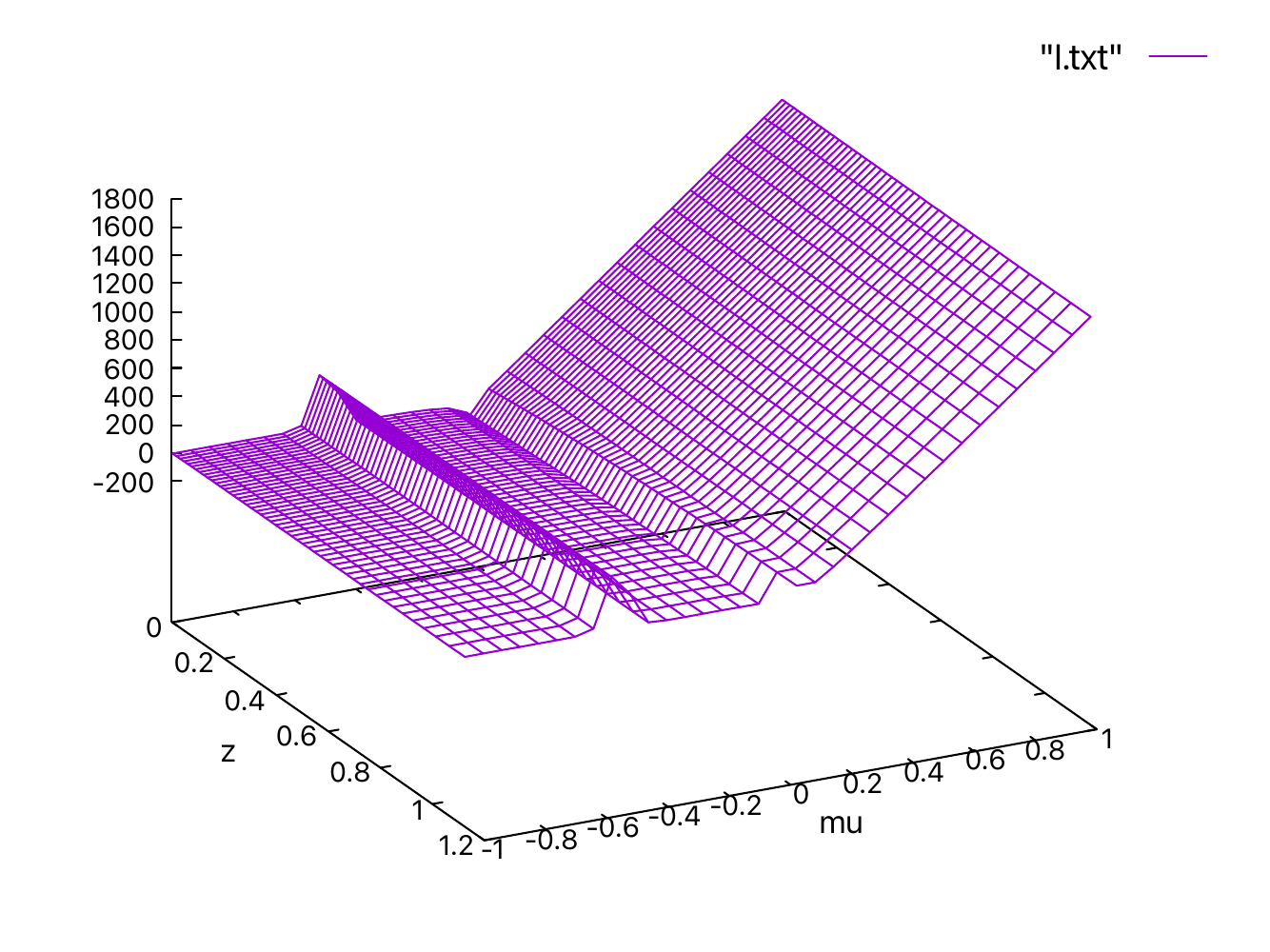}
\caption{\footnotesize \label{Isurf} $ \tilde I_\nu$ versus $z\in(0,Z)$ and $\mu\in(-1,1)$  at $\nu=4.2$. }
\end{center}
\end{minipage}
\hskip0.5cm
\begin{minipage}[b]{0.45\textwidth}
\begin{center}
\includegraphics[width=7cm]{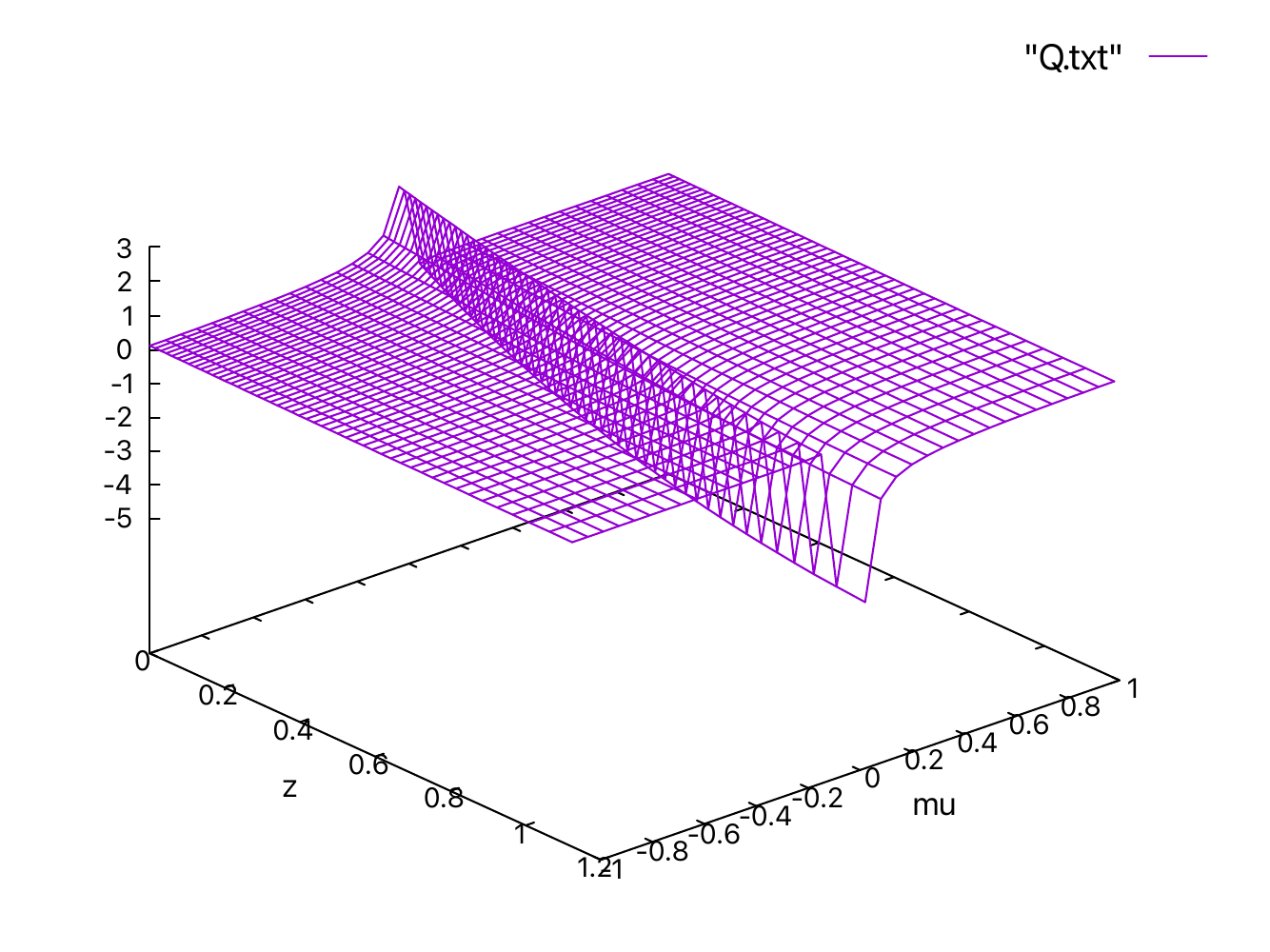}
\caption{\footnotesize \label{Qsurf} $ \tilde Q_\nu$ versus $z\in(0,Z)$ and $\mu\in(-1,1)$ at $\nu=4.2$.}
\end{center}
%%%%%%%%%%%
\end{minipage}
\end{figure}

\section*{Conclusion}
As far as the temperature in the atmosphere is concerned, fully collimated or semi-collimated sunlight on the tropopause give the same results.  Therefore one should ignore the azimuthal dependency and use semi-collimated light. To obtain the light intensity and the polarization for all azimuthal angles requires a small extra effort, namely the computation of two independent functions which can be obtained by the same algorithm, iterations on the source, hence a few extra lines in the computer program in the algorithmic loop.
\bibliographystyle{plain}
\bibliography{refbookRTE}

\end{document}